\documentclass[a4paper,12pt]{article}
\usepackage{latexsym}
\usepackage{amsmath}
\usepackage{amssymb}
\usepackage{enumerate}
\usepackage{mathrsfs}

\usepackage{theorem}
\newtheorem{theorem}{Theorem}[section]

\newtheorem{lemma}[theorem]{Lemma}

\newtheorem{thm}{Theorem}

\theorembodyfont{\rmfamily}
\newtheorem{proof}{\textmd{\textit{Proof.}}}

\newtheorem{remark}[theorem]{Remark}

\newtheorem{acknowledgement}{\textmd{\textit{Acknowledgements.}}}

\makeatletter

\@addtoreset{equation}{section}
\makeatother

\newcommand{\qedd}{\hfill \Box}

\newcommand{\N}{\ensuremath{\mathbb{N}}}
\newcommand{\R}{\ensuremath{\mathbb{R}}}
\newcommand{\cP}{\ensuremath{\mathcal{P}}}
\newcommand{\Span}[1]{\mathop{\mathrm{Span}}\langle{#1}\rangle}
\newcommand{\lr}[2]{\left\langle{#1},{#2}\right\rangle}

\def\spt{\mathop{\mathrm{supp}}\nolimits}

\makeatletter
 \long\def\@makefntext#1{\parindent 1em\noindent
  \hbox to 1.3em{\hss $^{\@thefnmark}$}#1}
\makeatother
\title{Equivalence of two orthogonalities\\ between probability measures}
\author{
Asuka Takatsu\thanks{Graduate School of Mathematics, Nagoya University, Nagoya 464-8602,
Japan ({\sf takatsu@math.nagoya-u.ac.jp}) \&
Max-Planck-Intitut f\"ur Mathematik, Vivatsgasse 7, 53111 Bonn, Germany.}}
\date{\empty}
\begin{document}

\maketitle

\begin{abstract}
Given any two probability measures on a Euclidean space with mean $0$ and finite variance, 
we demonstrate that the two probability measures are orthogonal in the sense of Wasserstein geometry if and only if 
the two spaces by spanned by the supports of each probability measure are orthogonal.
\footnote[0]{
{\bf keywords:} Wasserstein geometry; orthogonality.} 
\end{abstract}

\section{Introduction and Main theorem}
This paper is concerned with the two orthogonalities between a pair of probability measures on $\R^d$:
one is measured by the Wasserstein metric and
the other is  given in terms of the orthogonality between spaces spanned by the supports of probability measures.

Let $\cP_{2}$ be the set of Borel probability measures  $\mu$ on $\R^d$ with finite variance, namely  
\[
\int_{\R^d} |x|^2 d\mu(x) <\infty.
\]
Given any $\mu,\nu\in\cP_2$, we define their {\it $(L^2$-$)$Wasserstein distance} by 
\begin{equation}\label{eq:wass}
W_2(\mu,\nu):=\sqrt{\inf_{\sigma \in \Pi(\mu,\nu)} \int_{\R^d \times \R^d} |x-y|^2 d\sigma(x,y)},
\end{equation}
where $\Pi(\mu,\nu)$ is the set of probability measures $\sigma$ on $\R^d \times \R^d$ with marginals $\mu$ and $\nu$, 
that is,  $\sigma[B \times \R^d]=\mu[B]$ and $\sigma[\R^d \times B]=\nu[B]$ hold for all Borel sets $B \subset \R^d$.
The pair $(\cP_2,W_2)$ is a metric space and inherits several properties of $\R^d$ (for instance see~\cite[Section~6]{Vi2}).
For example, $(\cP_2,W_2)$ has a cone structure as well as $\R^d$.
We say that a metric space $(X,d_X)$ {\it has a cone structure} if there exists a metric space $(\Sigma, \angle)$ such that 
$(X,d_X)$ is isometric to the quotient space $(\Sigma \times [0,\infty)/\sim, d_C)$,  
where the equivalence relation $\sim$ is defined by $(\xi,s) \sim (\eta,t)$ if we have $(\xi,s) =(\eta,t)$ or $s = t = 0$, and the distance function $d_C$ is given by 
\[
d_C((\xi,s), (\eta,t)):= \sqrt{ s^2+t^2-2st \cos \left(\min\{\angle(\xi, \eta),\pi \} \right)}.
\]
Of course, $\R^d$ is isometric to the quotient space of  $(d-1)$-sphere with its standard metric and   
it was proved in~\cite{TY} that $(\cP_2,W_2)$ is isometric to the quotient space of  
\[
\cP_2:=\left\{\mu \in \cP_2 \, |\, W_2(\delta_0, \mu)=1  \right\}
\]
with the metric given by 
\[
\angle (\mu,\nu):=\arccos\left(1-\frac12W_2(\mu,\nu)^2 \right).
\]
We remark that 
$\angle (\mu,\nu)$  is regarded as the angle between the two Wasserstein geodesics from $\delta_0$ to $\mu$ and from $\delta_0$ to $\nu$, 
which in addition coincides with the comparison angle of $\angle\mu\delta_0\nu$ defined by the cosine formula of the form 
\[
\angle(\mu,\nu)
=\angle\mu\delta_0\nu
:= \arccos\left(\frac{W_2(\delta_0,\mu)^2+W_2(\delta_0,\nu)^2-W_2(\mu,\nu)^2 }{2W_2(\delta_0,\mu)W_2(\delta_0,\nu)}\right) .
\]
Note that this formula can be extended to any pair in $\cP_2\setminus\{\delta_0\}$.
%
Moreover, $\cP_2$ is isometric to the direct product of $\R^d$ and the convex subspace given by 
\[
\cP_{2,0}:=\left\{\mu \in \cP_2 \biggm| \int_{\R^d} x d\mu(x)=0 \right\}
\]
which also has a cone structure and contains no line, that is  its  vertex angle given by
\begin{equation*}
\sup\{ \angle(\mu,\nu)\, |\,  \mu,\nu  \in \cP_{2,0} \setminus \{\delta_0\} \} 
\end{equation*}
is less than $\pi$.
Indeed for any $\mu,\nu \in \cP_{2,0}\setminus \{\delta_0\}$,  
the marginals of the product measure $\mu \times \nu$ are obviously $\mu$ and $\nu$,   
and the fact  
\[
W_2(\delta_0,\mu)^2= \int_{\R^d} |x|^2 d\mu(x), \qquad W_2(\delta_0,\nu)^2= \int_{\R^d} |y|^2 d\nu(y)
\]
 together with the condition that the means of $\mu$ and $\nu$ are $0$ yield that 
\begin{equation*}
 W_2(\mu,\nu)^2  
\leq \int_{\R^d\times\R^d} |x-y|^2 d(\mu\times \nu) (x,y)
=W_2(\delta_0,\mu)^2+W_2(\delta_0,\nu)^2, 
\end{equation*}
which shows  
\begin{equation}\label{eq:2}
\angle(\mu,\nu)=\arccos\left(\frac{W_2(\delta_0,\mu)^2+W_2(\delta_0,\nu)^2-W_2(\mu,\nu)^2 }{2W_2(\delta_0,\mu)W_2(\delta_0,\nu)}\right)  \leq \frac\pi2.
\end{equation}
In the case of $d \geq 2$, the equality in~\eqref{eq:2} holds, for example, by taking  
$\mu:=(\delta_\xi + \delta_{-\xi})/2$ and $\nu:=(\delta_\eta + \delta_{-\eta})/2$ with $\xi,\eta \in \R^d\setminus\{0\}$ satisfying that $\lr{\xi}{\eta} =0$. 

We provide a necessary and sufficient condition for a pair in $\cP_{2,0}  \setminus\{\delta_0\}$ to attain the equality in~\eqref{eq:2}, in other words,  
for a pair of Wasserstein geodesics in $\cP_{2,0}$ stating from $\delta_0$  to be orthogonal.
For $A \subset \R^d$, let $\Span{A}$ denote  the linear span of $A$.
\begin{thm}
Given any $\mu,\nu \in \cP_{2,0}\setminus\{\delta_0\}$,  
the condition $\angle (\mu,\nu)=\pi/2$ is equivalent to the orthogonality between $\Span{\spt(\mu)}$ and $\Span{\spt(\nu)}$.
\end{thm}

\acknowledgement
The author would like to express her gratitude to C\'edric Villani for his valuable comments and discussions.
She is  grateful for the hospitality of Institut des Hautes \'Etudes Scientifiques, where most of this work was done.
She would also like to thank Takumi Yokota for his comments.
This work is partially supported by JSPS-IH\'ES (EPDI)  fellowship.
\section{Proof of Theorem}
We first prepare two lemmas derived from the feature  of mean.
\begin{lemma}\label{lem:mean}
For any $\mu \in\cP_{2,0}\setminus\{\delta_0\}$,  there exist $\{\xi_j\}_{j=0}^k \subset \spt(\mu)$ and $\{a_j\}_{j=0}^k \subset \R$ 
such that  $\sum_{j=0}^k a_j \xi_j=0$ and  $\sum_{j=0}^k a_j  \neq 0$.
\end{lemma}
\begin{proof}
Suppose that the dimension of $\Span{\spt(\mu)}$ equals $k$
and $\{\xi_i\}_{i=1}^k \subset \spt(\mu)$ is a basis of $\Span{\spt(\mu)}$.
For any orthonormal basis $\{u_i\}_{i=1}^k$ of $\Span{\spt(\mu)}$, there exists $\{a_{ij}\}_{1 \leq i, j \leq k} \subset \R$ such that $u_i=\sum_{j=1}^k a_{ij}\xi_j$.
We define the functions $p_i$ and $a_j$ on $\R^d$ by $p_i(x):=\lr{x}{u_i}$ and $a_j(x):= \sum_{i=1}^k p_i(x)  a_{ij} $.
If $\xi_0\in \spt(\mu)$ satisfies  $\sum_{j=1}^k  a_j(\xi_0)  \neq 1$, then $\{ \xi_j\}_{j=0}^k$ and  $\{ a_j (\xi_0) \}_{j=0}^k $ with $a_0(\xi_0)=-1$ are the desired families 
since we have 
\[
\xi_0
=\sum_{i=1}^k p_i(\xi_0) u_i
=\sum_{i=1}^k p_i(\xi_0)  \left(  \sum_{j=1}^k  a_{ij}  \xi_j \right) 
=\sum_{j=1}^k\left(  \sum_{i=1}^k p_i(\xi_0)  a_{ij} \right) \xi_j
=\sum_{j=1}^k a_j(\xi_0) \xi_j.
\]
Such a  point always exists, otherwise any $x\in \spt(\mu)$ satisfies $\sum_{j=1}^k a_j(x) = 1$ and we have a contradiction as   
\[
1
=\int_{\R^d} \sum_{j=1}^k  a_j(x)  d\mu(x)
=\int_{\R^d} \sum_{j=1}^k  \sum_{i=1}^k p_j(x) a_{ij} d\mu(x)
=\int_{\R^d} \lr{x}{ \sum_{j=1}^k  \sum_{i=1}^k a_{ij} u_i} d\mu(x)
=0,
\]
where the last equality follows from the definition of mean.
$\qedd$
\end{proof}
For a point $\xi \in \R^d$ and a family $\{ \Xi_\lambda\}_{\lambda \in \Lambda}$ of subsets in $\R^d$,  we set 
\begin{gather*}
\Xi_\lambda-\xi:=\{\xi_\lambda-\xi  \, | \, \xi_\lambda \in \Xi_\lambda \}, \qquad
\sum_{\lambda \in \Lambda} \Xi_\lambda:=\left\{ \sum_{i=1}^n \xi_{\lambda_i}  \Bigm| \xi_{\lambda_i} \in \Xi_{\lambda_i},\ \lambda_i \in \Lambda,\ \exists n \in \N \right\}.
\end{gather*}
\begin{lemma}\label{lemma3}
Given any $\mu \in\cP_{2,0} \setminus\{\delta_0\}$, we have 
\[
\sum_{\xi\in \spt(\mu)} \Span{\spt(\mu)-\xi}=\Span{\spt(\mu)}.
\]
\end{lemma}
\begin{proof}
Since the relation $\Span{\spt(\mu)-\xi} \subset \Span{\spt(\mu)} $ is trivially true for any $\xi \in \spt(\mu)$,   
the relation $ \sum_{\xi\in \spt(\mu)} \Span{\spt(\mu)-\xi} \subset \Span{\spt(\mu)}$ is also true.

Lemma~\ref{lem:mean} ensures the existences of $\{\xi_j\}_{j=0}^k \subset  \spt(\mu)$ and $\{a_j\}_{j=0}^k \subset \R$ such that $\sum_{j=0}^k a_j \xi_j=0$ and  $a:=\sum_{j=0}^k a_j \neq 0$. 
Then for any $x\in \Span{\spt(\mu)} $, we find that 
\[
x=\sum_{j=0}^k \frac{a_j}{a}(x-\xi_j) \in\sum_{\xi\in \spt(\mu)} \Span{\spt(\mu)-\xi} .
\]
$\qedd$
\end{proof}
Let us  now prove Theorem by using  the following known result.
\begin{lemma} {\rm(\cite[Theorem~5.10(ii)]{Vi2})}\label{mono}
Given any $\mu,\nu\in\cP_2$ and any $\sigma\in\Pi(\mu,\nu)$,  
the following two properties are equivalent to each other.\\
$(i)$ $\sigma$ attains the infimum in \eqref{eq:wass}.\\
$(ii)$ For any $n\in \N$ and $\{(x_i,y_i)\}_{i=1}^n \subset \spt(\sigma)$ with the convention $y_{n+1}=y_1$, 
we have 
\[
\sum_{i=1}^n |x_i-y_i|^2 \leq \sum_{i=1}^n |x_i-y_{i+1}|^2. 
\]
\end{lemma}
\noindent
{\it Proof of Theorem.}
Take any $\mu,\nu \in \cP_{2,0} \setminus\{\delta_0\}$ and fix them.
We first remark that $\mu \times \nu$ attains the infimum in~\eqref{eq:wass}  
if and only if $W_2(\mu,\nu)^2=W_2(\delta_0,\mu)^2+W_2(\delta_0,\nu)^2$ holds.

If  $\angle (\mu,\nu)=\pi/2$, then $\mu \times \nu$ attains the infimum in~\eqref{eq:wass} due to~\eqref{eq:2}.
For any $(x,y), (\xi,\eta) \in \spt(\mu \times \nu)$, $(x,\eta), (\xi, y)$ also lie in $\spt(\mu \times \nu)$ and Lemma~\ref{mono} yields that 
\[
|x-y|^2+|\xi-\eta|^2 \leq |x-\eta|^2+|\xi-y|^2 \leq |x-y|^2+|\xi-\eta|^2 
\]
meaning $\lr{x-\xi}{y-\eta}=0$.
Since we take $(x,y) \in\spt (\mu\times \nu)$ arbitrarily, the spaces $\Span{\spt (\mu)-\xi}$ and $\Span{\spt (\nu)-\eta}$ are orthogonal.
Moreover, Lemma~\ref{lemma3} with arbitrary choice of $(\xi,\eta) \in \spt (\mu\times\nu)$ provides the orthogonality between $\Span{\spt (\mu)}$ and $\Span{\spt (\nu)}$.

Conversely 
suppose the orthogonality between  $\Span{\spt (\mu)}$ and $\Span{\spt(\nu)}$. 
Then $\mu \times \nu \in \Pi(\mu,\nu)$ satisfies condition~(ii) in Lemma~\ref{mono} and thus $\mu\times\nu$ attains the infimum in~\eqref{eq:wass}, 
which in turn implies  
\[
\angle(\mu,\nu)
= \arccos\left(\frac{W_2(\delta_0,\mu)^2+W_2(\delta_0,\nu)^2-W_2(\mu,\nu)^2 }{2W_2(\delta_0,\mu)W_2(\delta_0,\nu)}\right) 
=\frac\pi2.
\]
$\qedd$

\begin{remark}
(1)The ``if" part can be proved in a different  way:
for any $\mu,\nu \in \cP_2\setminus\{ \delta_0\}$, let $\theta(\mu,\nu)$ be the  smallest principal angle between $\Span{\spt(\mu)}$ and  $\Span{\spt(\nu)}$, that is, 
\[
 \theta(\mu,\nu):=\min \left\{ \arccos \frac{ \lr{x}{y} }{ |x||y|}   \Bigm|  x \in \Span{\spt(\mu)},\ y \in \Span{\spt(\nu)} , \ x, y \neq 0  \right\}. 
\]
Note that $\theta(\mu,\nu) \in[0, \pi/2]$.
If  the relation $\angle (\mu,\nu) \geq \theta (\mu,\nu)$ holds,  
then the orthogonality between  $\Span{\spt(\mu)}$ and  $\Span{\spt(\nu)}$, that is $\theta (\mu,\nu)=\pi/2$, implies $\angle (\mu,\nu)=\pi/2$.

To prove the relation $\angle (\mu,\nu) \geq \theta (\mu,\nu)$, let $\sigma \in \Pi(\mu,\nu)$ be optimal.
Then  $\sigma$ is supported on $\Span{\spt(\mu)}\times\Span{\spt(\nu)}$ and H\"older's inequality yields that   
\begin{align}\label{eq:1}\notag
 W_2(\mu,\nu)^2  
%
&\geq W_2(\delta_0,\mu)^2+W_2(\delta_0,\nu)^2 -2 \cos \theta (\mu,\nu) \int_{\R^d \times \R^d} |x| |y| d \pi (x,y) \\ 
&\geq W_2(\delta_0,\mu)^2+W_2(\delta_0,\nu)^2 -2 \cos \theta(\mu,\nu)   W_2(\delta_0,\mu) W_2(\delta_0,\nu) ,
\end{align}
which shows $\angle (\mu,\nu) \geq \theta (\mu,\nu)$  as desired.
The relation $\angle (\mu,\nu) = \theta (\mu,\nu)$ holds if and only if all the inequalities in~\eqref{eq:1} are equalities, which is equivalent to the existence of $a \in \R$ 
such that $\lr{x}{y}=a|x|^2 \cos \theta(\mu,\nu)$ for $\sigma$-a.e.\ $(x,y) \in \R^d \times \R^d$.

(2)
In the case of $d=1$, 
Theorem yields that the angle between any pair in $\cP_{2,0} \setminus \{\delta_0\}$ is strictly less than $\pi/2$,
however the vertex angle equals $\pi/2$.
Indeed, set 
\begin{gather*}
\mu:=\frac12\left\{ \delta_1+\delta_{-1}\right\}, \quad
\nu_n:=\left(1-\frac{1}{n^2} \right) \delta_0 +\frac1{2n^2}\left\{\delta_n+\delta_{-n}\right\},\\
\sigma_n:=\frac12 \left(1-\frac{1}{n^2} \right)\left\{\delta_{(1,0)}+\delta_{(-1,0)}\right\}  +\frac1{2n^2}\left\{\delta_{(1,n)}+\delta_{(-1,-n)}\right\}
\end{gather*}
for any $n \in \N$.
Then we have $\mu,\nu_n \in \cP_{2,0}$ and $\sigma_n \in \Pi(\mu,\nu_n)$  attains the infimum in \eqref{eq:wass}, which implies 
\[
\angle(\mu,\nu_n)=\arccos\left(\frac{W_2(\delta_0,\mu)^2+W_2(\delta_0,\nu)^2-W_2(\mu,\nu)^2 }{2W_2(\delta_0,\mu)W_2(\delta_0,\nu)}\right)=\arccos\frac1n.
\]
\end{remark}


\end{document}